\documentclass[11pt, amsfonts]{amsart}

\usepackage{amsmath,amssymb,amsthm,color}
\usepackage[all]{xy}

\SelectTips{eu}{12}

\newtheorem{thm}{Theorem}
\newtheorem{prop}[thm]{Proposition}
\newtheorem{lem}[thm]{Lemma}
\newtheorem{cor}[thm]{Corollary}

\theoremstyle{definition}

\theoremstyle{remark}
\newtheorem{rem}[thm]{Remark}

\newcommand{\Aut}{{\rm Aut}}

\newcommand{\ZZ}{\mathbb{Z}}

\title{Graphs whose Kronecker covers are bipartite Kneser graphs}

\author{Takahiro Matsushita}
\address{Department of Mathematical Sciences, University of the Ryukyus, Nishihara-cho, Okinawa 903-0213, Japan}
\email{mtst@sci.u-ryukyu.ac.jp}

\subjclass[2010]{Primary 05C15; Secondary 55U10}

\keywords{Kneser graphs, bipartite Kneser graphs, Kronecker covers, automorphism groups of graphs, chromatic numbers, neighborhood complexes.}

\begin{document}

\baselineskip.525cm

\maketitle

\begin{abstract}
We show that there are $k$ simple graphs whose Kronecker covers are isomorphic to the bipartite Kneser graph $H(n,k)$, and determine the automorphism groups of these graphs. Using the neighborhood complexes of graphs, we also show that the chromatic numbers of these graphs coincide with $\chi(K(n,k)) = n - 2k + 2$.
\end{abstract}


\section{Introduction}

A {\it covering map} of graphs is a surjective graph homomorphism $p \colon \tilde{G} \to G$ such that $p$ maps the neighborhood of each vertex $v$ in $\tilde{G}$ bijectively onto the neighborhood of $p(v)$. The Kronecker cover of $G$ is the categorical product $K_2 \times G$ by $K_2$ (see Section 2). When $G$ is connected and non-bipartite, then its Kronecker cover is the unique double cover which is bipartite, and when $G$ is bipartite, its Kronecker cover is the disjoint union of two copies of $G$. Kronecker covers are fundamental objects in covering theory of graphs, and have appeared in different branches of combinatorics (see \cite{BGZ}, \cite{Lovasz}, \cite{Maldeghem}, \cite{Matsushita}, and \cite{Waller}).

It is known that different graphs may have isomorphic Kronecker covers. For example, Imrich and Pisanski \cite{IP} constructed a graph $G$ such that $G$ is not isomorphic to the Petersen graph but its Kronecker cover is isomorphic to the Kronecker cover of the Petersen graph. Then it is natural to classify all possible graphs whose Kronecker covers are isomorphic to a given bipartite graph. Such a problem was actually written in \cite{IP}, and was settled in the cases of hypercubes \cite{BIKZ} and generalized Petersen graphs \cite{KP}.

In general, graphs having the same double covers have many properties in common, and it is difficult to distinguish them. For example, they have the same degree sequence. Moreover, if $K_2 \times G \cong K_2 \times H$, then the neighborhood geometries of $G$ and $H$ coincide (see \cite{BGZ} and \cite{Maldeghem}). In other words, one cannot distinguish these graphs by the information of the family of sets which appear as a neighborhood of a vertex. 

The purpose in this paper is to classify the graphs whose Kronecker covers are isomorphic to the bipartite Kneser graph $H(n,k)$, and we further determine their automorphism groups and chromatic numbers. The bipartite Kneser graph is the Kronecker cover of the Kneser graph $K(n,k)$. Therefore we classify the graphs whose neighborhood geometries coincide with that of $K(n,k)$.

Now we recall the definitions of the Kneser graphs $K(n,k)$ and bipartite Kneser graphs $H(n,k)$. Let $n$ and $k$ be positive integers with $n > 2k$. The {\it Kneser graph $K(n,k)$} is the graph consisting of $k$-subsets of $[n] = \{ 1, \cdots, n\}$, where two $k$-subsets are adjacent if and only if they have no intersection. The graph $K(n,1)$ is the complete graph with $n$ vertices, and $K(5,2)$ is the Petersen graph. The Kneser graphs have appeared as fundamental objects in many branches of combinatorics (see \cite{GM}, \cite{Kozlov}, and \cite{SU} for example).

The {\it bipartite Kneser graph $H(n,k)$} is the Kronecker cover of $K(n,k)$. In many references (see \cite{Mirafzal} for example), $H(n,k)$ is formulated as follows: The vertex set of $H(n,k)$ is the family of subsets of $\{ 1, \cdots, n\}$ whose cardinality is $k$ or $n - k$. Two distinct vertices $\sigma$ and $\tau$ are adjacent if and only if one of them is included in the other. It is well-known that these two definitions are equivalent (see Lemma 20 of \cite{GP}). In this paper, to simplify the description, we consider $H(n,k)$ as the Kronecker cover of $K(n,k)$.

The bipartite Kneser graph $H(2k+1, k)$ is isomorphic to the graph $Q(2k+1, k)$ called the middle cube graph, and it was a long standing conjecture that the middle cube graph has a Hamiltonian cycle. This conjecture was settled by M\"utze \cite{Mutze}, in general M\"utze and Su \cite{MS} proved that the bipartite Kneser graphs are Hamiltonian. For more information about this subject, we refer to \cite{Mutze} and \cite{MS} and their references.

Kronecker covers have nice properties from the viewpoint of category theory. For general double covers, not every graph homomorphism $f \colon G \to G'$ has a lift. However, in the case of Kronecker covers, every graph homomorphism $f \colon G \to G'$ has a lift, and when we require a certain assumption, such a lift is uniquely determined (see Lemmas \ref{lem 5} and \ref{lem 6}). By these properties of Kronecker covers, we can determine the automorphism groups of graphs and the classification of the graphs having the same Kronecker covers from the automorphism groups of their Kronecker covers. Recently, the automorphism group of $H(n,k)$ was determined by Mirafzal \cite{Mirafzal}, and our computation is based on his result.

Now we are ready to state our results:

\begin{thm} \label{thm 1}
Suppose that $n > 2k$ and $k \ge 1$. Then there are $k$ simple graphs $$G_0(n,k), G_1(n,k), \cdots, G_{k-1}(n,k),$$
where $G_0(n,k) = K(n,k)$, $K_2 \times G_i(n,k) \cong H(n,k)$, but any two of them are not isomorphic.
\end{thm}

Here we recall some previous work concerning Theorem \ref{thm 1}. As was mentioned, Imrich and Pisanski \cite{IP} constructed a graph $G$ such that $G \not\cong K(5,2)$ but $K_2 \times G \cong H(5,2)$. Moreover, the author constucted a graph $KG'_{n,k}$ such that $KG'_{n,k} \not\cong K(n,k)$ but $K_2 \times KG'_{n,k} = H(n,k)$, when $k \ge 2$. In fact, the graph $KG'_{n,k}$ is the graph $G_1(n,k)$ in our sense, and $G_i(n,k)$ is its generalization. We also noted that the case $k=1$ of Theorem \ref{thm 1} is easily deduced by known methods (see Remark \ref{rem k=1}).

Next we study the automorphism groups of $G_i(n,k)$. Let $\ZZ_2$ denote the cyclic group of order $2$, and $S_n$ the symmetric group of $[n] = \{ 1, \cdots, n\}$. The automorphism group of $G_i(n,k)$ is described as follows. Here we recall that the automorphism group of $K(n,k)$ is isomorphic to $S_n$.

\begin{thm} \label{thm 2}
For $i = 0,1, \cdots, k-1$, there is a group isomorphism
$$\Aut(G_i(n,k)) \cong (\ZZ_2^i \rtimes_{\varphi} S_i) \times S_{n-2i}.$$
Here the action $\varphi \colon S_i \to \Aut(\ZZ_2^i)$ is defined by $\varphi(x_1, \cdots, x_i) = (x_{\sigma^{-1}(1)}, \cdots, x_{\sigma^{-1}(i)})$, and $\ZZ_2^i \rtimes_{\varphi} S_i$ is the semi-direct product of groups with respect to $\varphi$.
\end{thm}

Finally, we study the chromatic numbers of $G_i(n,k)$. Here we recall that $\chi(K(n,k)) = n - 2k + 2$ was known as the Kneser conjecture and was proved by Lov\'asz \cite{Lovasz}. In general $K_2 \times G \cong K_2 \times H$ does not imply $\chi(G) = \chi(H)$. For example, by an easy observation, we have
$$K_2 \times (K_2 \times K_n) \cong (K_2 \times K_n) \sqcup (K_2 \times K_n) \cong K_2 \times (K_n \sqcup K_n),$$
but $\chi(K_2 \times K_n) = 2$ and $\chi(K_n \sqcup K_n) = n$. Here the notation $K_n \sqcup K_n$ means that the union of two copies of $K_n$. Moreover, for given integers $m$ and $n$ greater than $2$, there are connected graphs $G$ and $H$ such that $\chi(G) = m$, $\chi(H) = n$, and $K_2 \times G \cong K_2 \times H$ (see \cite{Matsushita}). However, in the case of $G_i(n,k)$, we have that the chromatic numbers of these graphs coincide:

\begin{thm} \label{thm 3}
The chromatic number of $G_i(n,k)$ for $i = 0,1,\cdots, k-1$ coincides with $\chi(K(n,k)) = n-2k+2$.
\end{thm}
 In his outstanding proof of the Kneser conjecture, Lov\'asz introduced the neighborhood complex $N(G)$ of a graph $G$, and the connectivity of $N(G)$ gives a lower bound for the chromatic number $\chi(G)$ of $G$. Since the Kronecker cover of $G$ determines the isomorphism type of $N(G)$ (see \cite{Matsushita}) and Lov\'asz \cite{Lovasz} determined the connectivity of $N(K(n,k))$, we have the same lower bound for $\chi(G_i(n,k))$. This is the key observation of our proof of Theorem \ref{thm 3}.

The rest in this paper is organized as follows. In Section 2, we recall some terminology and facts concerning Kronecker coverings, and prove Theorem \ref{thm 1} and Theorem \ref{thm 2}. In Section 3, we review some facts of  neighbohood complexes and prove Theorem \ref{thm 3}.

\subsection*{Acknowledgement}
The author thanks the anonymous referees for their useful comments which significantly improved the presentation of this paper. The author is supported by JSPS KAKENHI 19K14536.

\section{Proofs of Theorem 1 and Theorem 2}

We first fix our notation and terminology, and review some facts concerning Kronecker covers. A graph is a pair $G = (V(G), E(G))$ consisting of a set $V(G)$ together with a symmetric binary relation $E(G)$ of $V(G)$. We write $v \sim_G w$ or simply $v \sim w$ to mean that $v$ and $w$ are adjacent in $G$. A map $f \colon V(G) \to V(H)$ is a {\it graph homomorphism} if $v \sim_G w$ implies $f(v) \sim_H f(w)$. An {\it $n$-coloring} is a graph homomorphism from $G$ to $K_n$. An isomorphism is a graph homomorphism having an inverse which is a graph homomorphism. An {\it automorphism of $G$} is an isomorphism from $G$ to $G$. Let $\Aut(G)$ denote the automorphism group of $G$. An {\it involution of $G$} is an automorphism $\alpha$ of $G$ such that $\alpha^2 = {\rm id}_G$.

A {\it bigraph} \cite{BGZ} is a graph $X$ equipped with a $2$-coloring $\varepsilon \colon G \to K_2$. For a pair $X$ and $Y$ of bigraphs, a graph homomorphism $f \colon V(X) \to V(Y)$ is {\it even} if $\varepsilon f = \varepsilon$, and {\it odd} if $\varepsilon f(x) \ne \varepsilon (x)$ for every $x \in V(X)$.

Let $G$ and $H$ be graphs. The {\it categorical product $G \times H$} is the graph whose vertex set is $V(G) \times V(H)$, and $(v,w) \sim_{G \times H} (v', w')$ if and only if $v \sim_G v'$ and $w \sim_H w'$. The {\it Kronecker cover} is the categorical product $K_2 \times G$. Note that the Kronecker cover $K_2 \times G$ has a $2$-coloring $K_2 \times G \to K_2$, $(i,v) \mapsto i$, and has an odd involution $(1, v) \leftrightarrow (2,v)$. In fact, every bigraph $X$ equipped with an odd involution $\alpha$ is isomorphic to the Kronecker cover over a certain graph $X / \alpha$ defined as follows.

Let $X$ be a bigraph with an odd involution $\alpha$. Define the quotient graph $X/\alpha$ by
$$V(X/\alpha) = \big\{ \{ x, \alpha (x)\} \; | \; x \in V(X) \big\},$$
$$E(X/\alpha) = \{ (\sigma, \tau) \in V(X/\alpha) \times V(X/\alpha) \; | \; (\sigma \times \tau) \cap E(X) \ne \emptyset \}.$$
In other words, $\sigma \sim_{X/\alpha} \tau$ if and only if there is $x \in \sigma$ and $y \in \tau$ such that $x \sim_X y$. Note that $X/\alpha$ is not simple in general. In fact, $X/\alpha$ is simple if and only if there is no vertex $x$ in $X$ such that $x \sim_X \alpha(x)$. The graph homomorphism
$$(\varepsilon, \pi) \colon X \to K_2 \times (X/\alpha), x\mapsto(\varepsilon(x), \pi(x))$$
is an even isomorphism (see \cite{Matsushita}).

\begin{rem} \label{rem k=1}
Let $n$ be an integer greater than $2$. For each $v \in K_2 \times K_n$, there is only one vertex $w$ such that the distance between $v$ and $w$ is odd and greater than $2$. This means that there is only one odd involution $\alpha$ on $K_2 \times K_n$ such that $(K_2 \times K_n)/ \alpha$ is simple. Of course, it is $K_n$.
\end{rem}

The following two lemmas are known (see \cite{IP} and Theorem 3.1 of \cite{Matsushita}) and easily proved.

\begin{lem} \label{lem 5}
Let $X$ and $Y$ be bigraphs, $\alpha$ and $\beta$ odd involutions of $X$ and $Y$ respectively, and $f \colon X \to Y$ a (not necessarily even) graph homomorphism satisfying $f \alpha = \beta f$. Then there is a unique graph homomorphism $\overline{f} \colon X/\alpha \to Y/\beta$ satisfying $\pi f = \overline{f} \pi$. If $f$ is an isomorphism, then $\overline{f}$ is an isomorphism.
\end{lem}

\begin{lem} \label{lem 6}
Let $X$ and $Y$ be bigraphs, and $\alpha$ and $\beta$ odd involutions of $X$ and $Y$, respectively. For every graph homomorphism $f \colon X/\alpha \to Y/\beta$, there is a unique even graph homomorphism $\tilde{f} \colon X \to Y$ such that $\pi \tilde{f} = f \pi$. If $f$ is an isomorphism, then $\tilde{f}$ is also an isomorphism.
\end{lem}

Here we mention two important applications of these lemmas. Two odd involutions $\alpha$ and $\beta$ of $X$ are {\it conjugate} if there is an automorphism $f$ such that $f \alpha = \beta f$. Similarly, $\alpha$ and $\beta$ are {\it evenly conjugate} if there is an even automorphism $f$ such that $f \alpha = \beta f$. Using this terminology, we have the following classification result. Note that in the following corollary, the implication $(3) \Rightarrow (1)$ is known (see Proposition 3 of \cite{IP} for example).

\begin{prop} \label{prop 7}
Let $\alpha$ and $\beta$ be odd involutions in a bigraph $X$. Then the following are equivalent.
\begin{itemize}
\item[(1)] $X/\alpha$ and $X/\beta$ are isomorphic.

\item[(2)] $\alpha$ and $\beta$ are evenly conjugate.

\item[(3)] $\alpha$ and $\beta$ are conjugate.
\end{itemize}
\end{prop}
\begin{proof}
Suppose that there is an isomorphism $f \colon X/\alpha \to X/\beta$. It follows from Lemma \ref{lem 6} that there is an even isomorphism $\tilde{f} \colon X \to X$ satisfying $\tilde{f} \alpha = \beta \tilde{f}$. It is clear that (2) implies (3). It follows from Lemma \ref{lem 5} that (3) implies (1).
\end{proof}

\begin{prop} \label{prop 8}
Let $\alpha$ be an odd involution of a bigraph $X$. Then $\Aut(X/\alpha)$ is isomorphic to the subgroup of $\Aut(X)$ consisting of even elements commuting with $\alpha$.
\end{prop}
\begin{proof}
Let $\Gamma$ be the subgroup of $\Aut(X)$ consisting of even elements commuting with $\alpha$. Define the group homomorphisms $\Phi \colon \Gamma \to \Aut(X/\alpha)$ and $\Psi \colon \Aut(X/\alpha) \to \Gamma$ as follows. Let $f \in \Gamma$. Since $f \alpha = \alpha f$, Lemma \ref{lem 5} implies that $f$ induces an isomorphism $\Phi(f) = \overline {f} \colon X/\alpha \to X/\alpha$. On the other hand, let $g \in \Aut(X/\alpha)$. It follows from Lemma \ref{lem 6} that there is a unique even automorphism $\tilde{g} \colon X \to X$ satisfying $\tilde{g} \alpha = \alpha \tilde{g}$, and put $\Psi(g) = \tilde{g} \in \Gamma$. These correspondences are group homomorphisms and $\Psi$ is the inverse of $\Phi$.
\end{proof}

Now we study the automorphism group of $K_2 \times G$. For a pair of graphs $G$ and $H$, we have a monomorphism $\Aut(G) \times \Aut(H) \to \Aut(G \times H)$ which sends $(f,g)$ to $f \times g$. Here $f \times g$ is the automorphism sending $(v,w)$ to $(f(v), g(w))$. Since $\Aut(K_2) = \ZZ_2$, there is a monomorphism
\begin{align}
\ZZ_2 \times \Aut(G) \to \Aut(K_2 \times G). \tag{$*$}
\end{align}
In general, this monomorphism is not an isomorphism (see Remark \ref{rem 1} for example). However, when $G = K(n,k)$, this monomorphism is an isomorphism:

\begin{thm}[Mirafzal \cite{Mirafzal}]
If $n > 2k$, then, the group homomorphism
$$\ZZ_2 \times \Aut(K(n,k)) \to \Aut(H(n,k))$$
described in $(*)$ is an isomorphism. In particular, $\Aut(H(n,k)) \cong \ZZ_2 \times S_n$.
\end{thm}

When the monomorphism $(*)$ is an isomorphism, then the classification of the graphs whose Kronecker covers are $K_2 \times G$ is simpler. Here we write $\tau$ to indicate the non-trivial involution of $K_2$.

\begin{prop}
Let $G$ be a graph and suppose that the monomorphism $(*)$ is an isomorphism. Then the following hold:
\begin{itemize}
\item[(1)] For every odd involution $\alpha$ of $K_2 \times G$, there is an involution $\alpha'$ of $G$ with $\alpha = \tau \times \alpha'$.
\item[(2)] Let $\alpha'$ and $\beta'$ be involutions of $G$. Then $\tau \times \alpha'$ and $\tau \times \beta'$ are evenly conjugate if and only if $\alpha'$ and $\beta'$ are conjugate, i.e. there is $f \in \Aut(G)$ with $f \alpha' = \beta' f$.
\end{itemize}
\end{prop}
\begin{proof}
Since the monomorphism $(*)$ is an isomorphism, every involution $\alpha$ of $K_2 \times G$ is written by ${\rm id}_{K_2} \times \alpha'$ or $\tau \times \alpha'$ for some $\alpha' \in \Aut(G)$. Since $\alpha$ is an involution, we have that $\alpha'$ is an involution. The involution ${\rm id}_{K_2} \times \alpha'$ is even and $\tau \times \alpha'$ is odd. This proves (1). (2) follows from the fact that every even automorphism of $K_2 \times G$ is written as ${\rm id}_{K_2} \times f$ for some $f \in \Aut(G)$ under our assumption.
\end{proof}

We are now ready to prove Theorem \ref{thm 1}.

\vspace{2mm}
\noindent
{\it Proof of Theorem \ref{thm 1}.} For $i = 0, 1, \cdots, [n/2]$, define $\sigma_i \in S_n$ to be the composite of transpositions
$$(1,2)(3,4) \cdots (2i-1,2i).$$
By the classification of conjugacy classes of $S_n$, every element in $S_n$ of order $2$ is conjugate to some $\sigma_i$, and $i \ne j$ implies that $\sigma_i$ and $\sigma_j$ are not conjugate. Then $\tau \times \sigma_i$ is an odd involution of $K_2 \times G$, and put $G_i = G_i(n,k) = H(n,k) / (\tau \times \sigma_i)$. Then $G_0 = K(n,k)$, $i \ne j$ implies $G_i \ne G_j$, and for every odd involution $\alpha$ of $H(n,k)$, there is $i$ with $G_i \cong H(n,k) / \alpha$. To complete the proof, we prove that $G_i$ is simple if and only if $i < k$.

Suppose $i \ge k$. Then put $v = \{ 1,3, \cdots, 2k-1\} \in V(K(n,k))$. Then $v \sim \sigma_i (v)$ in $K(n,k)$ implies that $(1,v) \sim (2, \sigma_i(v)) = (\tau \times \sigma_i)(1,v)$ in $H(n,k)$. Thus $G_i$ is not simple. On the other hand, suppose $i < k$. Then for each $v \in V(K(n,k))$, we have that $\sigma_i(v) \cap v \ne \emptyset$ and hence $v \not\sim \sigma_i(v)$ for every $v \in V(K(n,k))$. This means that $(i,v) \not\sim (\tau(i), \sigma_i(v))$ for every $(i, v) \in V(H(n,k))$. This means that $G_i$ is simple, and completes the proof. \qed

\vspace{2mm}
Before giving the proof of Theorem \ref{thm 2}, we introduce the following notation: Let $G$ be a group and $x$ an element in $G$. We write $Z_G(x)$ to indicate the subgroup of $G$ consisting of the elements in $G$, which commute with $x$.

\vspace{2mm} \noindent
{\it Proof of Theorem \ref{thm 2}.}
Proposition \ref{prop 8} implies that the automorphism group of $G_i(n,k) = X / (\tau \times \sigma_i)$ is isomorphic to $Z_{S_n}(\sigma_i)$. Hence the following proposition completes the proof. \qed

\begin{prop}
For $m = 0,1,\cdots, [n/2]$, there is a following isomorphism:
$$Z_{S_n}(\sigma_m) = (\ZZ_2^m \rtimes_{\varphi} S_m) \times S_{n-2m}$$
Here the action $\varphi \colon S_m \to \Aut(\ZZ_2^m)$ is defined by
$$\varphi(\sigma)(x_1, \cdots, x_m) = (x_{\sigma^{-1}(1)}, \cdots, x_{\sigma^{-1}(m)}).$$
\end{prop}
\begin{proof}
Let $\sigma$ be an element of $Z_{S_n}(\sigma_m)$. Then $\sigma$ does not send an element of $\{ 1, \cdots, 2m\}$ to $\{ 2m+1, \cdots, n\}$. This means $Z_{S_n}(\sigma_m) \cong Z_{S_{2m}}(\sigma_m) \times S_{n-2m}$. In $S_{2m}$, $\sigma_m$ is conjugate with the element
$$\tau = (1,m+1) \cdots (m,2m).$$
Thus it suffices to show $Z_{S_{2m}}(\tau) = \ZZ_2^m \rtimes_\varphi S_m$.

First we define the group homomorphism $\Phi \colon \ZZ_2^m \rtimes_\varphi S_m \to Z_{S_{2m}}(\tau)$. For $i = 1, \cdots, m$, set $\varepsilon_i = (i, n+i) \in S_{2m}$. For $\sigma \in S_m$, define $\tilde{\sigma} \in S_{2m}$ by
$$\tilde{\sigma}(i) = \begin{cases}
\sigma(i) & (i = 1, \cdots, m)\\
\sigma(i-m) + m & (i = m+1, \cdots, 2m).
\end{cases}$$
Let $\Phi \colon \ZZ_2^m \rtimes_\varphi S_m \to Z_{S_{2m}}(\tau)$ be the map which sends $((x_1, \cdots, x_m), \sigma)$ to $\varepsilon_1^{x_1} \cdots \varepsilon_m^{x_m} \tilde{\sigma}$. Using the relation $\varepsilon_i \tilde{\sigma} = \tilde{\sigma} \varepsilon_{\sigma^{-1}(i)}$, we have that $\Phi$ is a group homomorphism.

Since $\Phi$ is injective, it suffices to show that $\Phi$ is surjective. Let $\sigma \in Z_{2m}(\tau)$. We identify $\{ 1, \cdots, 2n\}$ with $\ZZ_{2n}$, and for $i = 1, \cdots, 2n$, define $k_i \in \ZZ_{2n}$ by $\sigma(i) = i + k_i$. Since $\sigma$ and $\tau$ commute, we have $k_i = k_{n+i}$. This means that $\sigma$ gives rise to a permutation of the family of sets
$$s_1 = \{ 1, n+1\}, s_2 = \{ 2, n+2\}, \cdots, s_n = \{ n, 2n\}.$$
Define $\sigma' \in S_m$ by $\sigma(s_i) = s_{{\sigma'}(i)}$. For $i = 1, \cdots, m$, define $x_i \in \ZZ_2$ as follows:
\begin{itemize}
\item If $\sigma(i) = \sigma'(i)$, then $x_{\sigma(i)} = 0$.
\item If $\sigma(i) = \sigma'(i) + n$, then $x_{\sigma(i)} = 1$.
\end{itemize}
Then we have $\Phi((x_1, \cdots, x_m), \sigma') = \sigma$. This completes the proof.
\end{proof}

\begin{rem}\label{rem 1}
If $i > 0$, the group homomorphism
$$\ZZ_2 \times (\ZZ_2^i \rtimes_\varphi S_i) \times S_{n-2i} \cong \ZZ_2 \times \Aut(G_i(n,k)) \to \Aut(H(n,k)) \cong \ZZ_2 \times S_n$$
described by $(*)$ is not an isomorphism.
\end{rem}

\section{Proof of Theorem \ref{thm 3}}

The purpose in this section is to prove Theorem \ref{thm 3}. Namely, we want to show that $\chi(G_i(n,k))) = n - 2k + 2$ if $n > 2k$. We note that the proof given here is a straightforward generalization of the proof of $\chi(G_1(n,k)) = n - 2k + 2$ in \cite{Matsushita}.

Recall that Lov\'asz \cite{Lovasz} introduced neighborhood complexes of graphs to determine $\chi(K(n,k))$. We first review the definition and facts concerning neighborhood complexes. Let $G$ be a graph. Then the neighborhood complex $N(G)$ is the simplicial complex whose simplex is a subset of $V(G)$ having a common neighbor. Lov\'asz showed the following two theorems in his proof of Kneser's conjecture:

\begin{thm} \label{thm 3.1}
If $N(G)$ is $m$-connected, then $\chi(G) \ge m + 3$.
\end{thm}

\begin{thm} \label{thm 3.2}
The neighborhood complex $N(K(n,k))$ of $K(n,k)$ is $(n-2k-1)$-connected.
\end{thm}

On the other hand, the author noted in \cite{Matsushita} that the Kronecker cover of a graph $G$ determines the neighborhood complex of $G$:

\begin{lem}[Theorem 1.2 of \cite{Matsushita}. See also \cite{BGZ}] \label{lem 3.3}
Let $G$ and $H$ be graphs. If $K_2 \times G \cong K_2 \times H$, then their neighborhood complexes $N(G)$ and $N(H)$ are isomorphic.
\end{lem}

Combining the above results, we have the following corollary:

\begin{cor}
The neighborhood complex $N(G_i(n,k))$ is $(n-2k-1)$-connected. In particular, the inequality $\chi(G_i(n,k)) \ge n - 2k + 2$ holds.
\end{cor}

We now complete the proof of $\chi(G_i(n,k)) = n - 2k + 2$. This is proved by induction on $n$. First, note that $G_i(2k,k)$ is a disjoint union of copies of $K_2$, and hence it is clear that $\chi(G_i(2k,k)) = 2$. Suppose that $n > 2k$ and $\chi(G_i(n-1,k)) = n-2k+1$. A vertex in $G_i(n,k)$ which is not contained in $G_i(n-1,k)$ is written as
$$\{ (1, s), (2, \sigma_i s)\},$$
where $s$ is a $k$-subset of $[n]$ containing $n$. Note that $\sigma_i \in S_n$ fixes $n$. Since $G_i(n-1,k)$ is an induced subgraph of $G_i(n,k)$, we have that
$$n-2k+2 \le \chi(G_i(n,k)) \le \chi(G_i(n-1,k)) + 1 = n-2k+2.$$
This completes the proof.

\end{document}